\documentclass[12pt]{amsart}
\usepackage{amscd,amsxtra,amssymb,mathrsfs, bbm,url}
\usepackage{amsmath}

\usepackage[paper=a4paper,left=25mm,right=25mm,top=30mm,bottom=30mm]{geometry}
\usepackage{tikz-cd}
\theoremstyle{plain}
\newtheorem{thm}{Theorem}[section]

\newtheorem{lem}[thm]{Lemma}
\newtheorem{prop}[thm]{Proposition}
\newtheorem{cor}[thm]{Corollary}
\newtheorem{defn}[thm]{Definition}
\theoremstyle{definition}

\newtheorem{ex}[thm]{Example}

\theoremstyle{remark}
\newtheorem{rem}[thm]{Remark}

\begin{document}

\title{Harder-Narasimhan Filtration on Moment Map for Quiver Representations}
\author{Ching Yan Timothy Yau}
\date{\today}

\maketitle

\begin{abstract}
For a quiver without loops with many edges, we generalize the methods of Kac, Crawley-Boevey and Reineke and compute the dimension of Harder-Narasimhan strata of the zero set of the moment map. We notice a link between this dimension and the terms in the Kac's polynomial, which is given by the equivariant cohomology of this zero set.
\end{abstract}

\section{Introduction}
\hspace{-12pt}Harder-Narasimhan stratification, originally defined for vector bundles, was defined on the moduli space $Rep(Q,\alpha)$ of quiver representations by Reineke in \cite{reineke2003hn}. They were shown in \cite{hoskins2014stratifications},\cite{hoskins2018stratifications} to coincide with the Hesselink strata for the action of the general linear group on the same representation space. In general, if we have a free action of a reductive group $G$ on a smooth algebraic variety $X$, then these strata contain a lot of information about the equivariant cohomology $H_G^{\bullet}(X)$ (cf. \cite{kirwan1984cohomology}). 
\\~\\ Meanwhile, Kac's polynomial was introduced in \cite{kac1983root} and counts the number of absolutely indecomposable representations of a quiver of fixed dimension over finite fields. For indivisible dimension vectors, it was shown by Crawley-Boevey and Van den Bergh that \cite{crawley2004absolutely} that the coefficients of Kac's polynomial can be expressed as the dimension of singular cohomology groups of the Nakajima's quiver variety with zero framing. This variety is given by the GIT quotient of the semistable part of the zero level set of the moment map $\mu$. Therefore, its singular cohomology is equal to the equivariant cohomology $H_G(\mu^{-1}(0)^{ss})$. 
\\~\\ We recall that if the semistable stratum is nonempty, then the Harder-Narasimhan stratification partitions the moduli space of quiver representations into an open semistable stratums and other smaller stratum. The formula given in \cite{kirwan1984cohomology} is based on computing the codimension of these smaller stratum. For the moduli space $Rep(Q,\alpha)$, such codimension is computed in \cite{reineke2003hn}. We can define a stratification on $\mu^{-1}(0)$ by restricting the stratification of $Rep(\overline{Q},\alpha)$.
\\~\\Given a quiver $Q=(I,\Omega)$, we can consider the graph $Q_{\underline{n}}$ where $\underline{n}\in\mathbb N_{\geq 1}^{\Omega}$, and $Q_{\underline{n}}$ is obtained by replacing an edge $h:i\to j\in\Omega$ with $\underline{n}_h$ identical edges. When $\underline{n}$ is sufficiently large, we provide a formula for the dimension of the Harder-Narasimhan strata on $\mu^{-1}(0)$, when the quiver has sufficiently many edges (Theorem 4.6, Corollary 5.5). The same idea was pursued in \cite{hennecart2021asymptotic}. In that paper, Hennecart showed, using a formula given in \cite{hua2000counting}, that the coefficients of the Kac's polynomial stabilize. In fact, by the formula, the Kac's polynomial can be written naturally as a sum of polynomials $$\sum_jq^{l_j(\underline{n})}P_j(q)/Q_j(q)$$ indexed by the Harder-Narasimhan stratum. Here $l_j(\underline{n})$ is a linear form in $\underline{n}$ and $P_j,Q_j$ are polynomials independent of $\underline{n}$. The second main result of this paper asserts that these linear forms can match the dimension of the Harder-Narasimhan stratum. 
\\~\\ Section 2 of this paper recalls the basic notion and results in quiver representations, Harder-Narasimhan stratifications and Kac's polynomial. 
\\~\\ In Section 3,4 we generalize the method in \cite{crawley1998noncommutative}, \cite{crawley2001geometry},\cite{reineke2003hn}. We prove an exact sequence analogous to that in \cite[Section 3]{crawley2001geometry} that will allow us to compute the dimension of $\mu^{-1}(0)$.
\\~\\ In Section 5, we provide a weaker bound on the number of edges for which our dimension formula holds. We also compare the codimension of the Harder-Narasimhan stratum of $\mu^{-1}(0)$ and $Rep(\overline{Q},\alpha)$. We deduce that they differ by a constant that is independent of $\underline{n}$.
\\~\\ In Section 6, we recall Hua's formula and relate it to the codimension formula we obtained in the previous section.
\section{Background}
\hspace{-13pt}Throughout the entire paper, we will consider a quiver $Q=(I,\Omega)$, where $I$ is the set of vertices and $\Omega$ is the set of edges. If $h\in\Omega$, we let $t(h)$ denote its tail and $s(h)$ denote its head. Suppose $v,w\in\mathbb Z^I$, we define the Euler form $\langle\bullet,\bullet\rangle$ by
$$\langle v,w\rangle=\sum_{i\in I}v_iw_i-\sum_{h\in\Omega}v_{s(h)}w_{t(h)}$$
and define the symmetrized Euler form $(\bullet,\bullet)$ by $(v,w)=\langle v,w\rangle+\langle w,v\rangle$.

\subsection{Quiver Varieties}
We use \cite{kirillov2016quiver} as a reference for this subsection.
Let $\alpha\in\mathbb Z^I$, the set of complex representations of a quiver $Q$ with dimension vector $\alpha$ will be denoted by $Rep(Q,\alpha)$. It has a natural structure of a vector space. The group 
$$GL_{\alpha}(\mathbb C)=\prod_{i\in I}GL_{\alpha_i}(\mathbb C)$$
acts on the space $Rep(Q,\alpha)$ by conjugation. This action factors through the subgroup $\mathbb C^{\times}\subset GL_{\alpha}(\mathbb C)$ that consists of the same multiple of identity at each entry. Now, suppose $\theta\in\mathbb Z^{I}$ and fix a character $\chi_{\theta}$ of $GL_{\alpha}(\mathbb C)$ which must be of the form 
$$g\mapsto \prod\det(g_i)^{-\theta_i}$$
Assume that $\theta\cdot\alpha=0$, so that $\theta$ can be treated as a character of $PGL_{\alpha}(\mathbb C)=GL_{\alpha}(\mathbb C)/\mathbb C^{\times}$. Using standard constructions in  geometric invariant theory, we can form the categorical quotient
$$\mathcal{R}_0(\alpha)=Rep(Q,\alpha)/\!/PGL_{\alpha}(\mathbb C)$$
and the GIT quotient
$$\mathcal{R}_{\theta}(\alpha)=Rep(Q,\alpha)/\!/_{\chi_{\theta}}PGL(v)$$
We recall the following definition of semistability as defined in \cite{king1994moduli}
\begin{defn}
    Let $\theta\in\mathbb R^{I}$, we say that a representation $V\in Rep(Q,\alpha)$ is $\theta$-semistable if for every nontrivial subrepresentation $W$ of $V$, 
    $$\dim W\cdot \theta\leq 0$$
    and $\theta$-stable if the above inequality is strict.
\end{defn}
\begin{thm}
    \cite{king1994moduli} Under the action of $GL_{\alpha}(\mathbb C)$, points of $\mathcal{R}_0(\alpha)$ corresponds to orbits of semisimple representations and points of $\mathcal{R}_{\theta}(\alpha)$ corresponds to orbits of $\theta$-semistable representations.
\end{thm}
\hspace{-12pt}Let $\overline{Q}$ be the double quiver of $Q$. That is, $Q$ has vertex set $I$ and edge set $\Omega\cup \Omega^{op}$, where $\Omega^{op}$ consists of edges in $\Omega$ but in the opposite direction. We also let $Q^{op}$ to be the quiver with vertex set $I$ and edge set $\Omega^{op}$. We notice that 
$$Rep(\overline{Q},\alpha)\cong T^*Rep(Q,\alpha)$$
Let $\displaystyle \mathfrak{g}_{\alpha}=\bigoplus_{i\in I}\mathfrak{gl}_{\alpha_i}(\mathbb C)$. It is well-known that the action of $GL_{\alpha}(\mathbb C)$ on $Rep(Q,\alpha)$ is Hamiltonian with moment map 
$$\mu:Rep(\overline{Q},\alpha)\to\mathfrak{g}_{\alpha}^*$$
$$(x_h)_{h\in\Omega}\mapsto \sum_{h\in\Omega}(x_hx_{\overline{h}}-x_{\overline{h}}x_h)$$
here $\overline{h}$ is the opposite edge of $h$.
\begin{defn}
    Suppose $\theta\in\mathbb Z^I$, define the quiver variety by
    $$\mathcal{M}^{\theta}_{\alpha}(Q)=\mu^{-1}(0)/\!\!/_{\chi_{\theta}}GL_{\alpha}(\mathbb C)$$
\end{defn}
\begin{rem} In \cite{nakajima1994instantons}, \cite{nakajima1998quiver}, Nakajima defined a more general notion of quiver varieties with framing. Definition 2.3 corresponds to Nakajima's notion when the framing is zero.
\end{rem}
\hspace{-14pt}We recall also the main result in \cite{crawley2001geometry}. It gives the dimension of $\mu^{-1}(0)$ under some conditions.
\begin{thm}
    \cite[Theorem 1.1]{crawley2001geometry} Let $p(\theta)=1-\langle\theta,\theta\rangle$ for all $\theta\in\mathbb Z^I$, then if $\displaystyle p(\alpha)\geq\sum_{t=1}^rp(\beta^{(t)})$ for any decomposition $\alpha=\beta^{(1)}+...+\beta^{(r)}$ with the $\beta^{(t)}$ positive roots, then $\mu$ is flat and $\mu^{-1}(0)$ has dimension 
    $$\alpha\cdot\alpha-1+2p(\alpha)$$
    here $\alpha\cdot\alpha=\sum\alpha_i^2$.
\end{thm}
\hspace{-13pt}In section 3 and 4 we will generailze the method of proving this result.
\subsection{Harder-Narasimhan Filtration}
Reineke introduced Harder-Narasimhan filtration on quiver representations in \cite{reineke2003hn}. Again let $\theta\in\mathbb Z^I$ and suppose $k$ is any field. For a $k$-representation $V$ of $Q$ with dimension vector $\beta$, define
$$\mu(V)=\frac{\displaystyle\sum_{i\in I}\beta_i\theta_i}{\displaystyle \sum_{i\in I}\beta_i}$$
In particular if $\theta\cdot\alpha=0$ then $\mu(V)=0$ for all $V\in Rep(Q,\alpha)$, and $V$ is $\theta$-semistable if and only if $\mu(W)\leq 0$ for all proper subrepresentation of $V$.
\begin{defn}\cite[Definition 2.4]{reineke2003hn} Let $V$ be a representation of $Q$. A Harder-Narasimhan filtration of $X$ is a filtration $0=V_0\subset ...\subset V_s=V$ such that the quotient $V_i/V_{i-1}$ are semistable for $i=1,...,s$ and that 
$$\mu(V_1/V_0)>\mu(V_2/V_1)>...>\mu(V_s/V_{s-1})$$
\end{defn}
\begin{thm}
    \cite[Proposition 2.5]{reineke2003hn} Any representation of $Q$ possess a unique Harder-Narasimhan filtration.
\end{thm}
\hspace{-13pt}The type of a representation $V$ is the tuple $d^*=(\dim V_1/V_0,...,\dim V_s/V_{s-1})$. We define $Rep(Q,\alpha)_{d^*}^{HN}$ to be the subset of representations whose Harder-Narasimhan filtration is of type $d^*$. The weight of $d^*$ is the dimension vector of $V$. The main result in \cite{reineke2003hn} states that
\begin{thm}
    \cite[Proposition 3.4]{reineke2003hn} We say that $d^*$ with weight $\alpha$ is a HN type if $Rep(Q,\alpha)_{d^*}^{HN}$ is nonempty. The Harder-Narasimhan strata $Rep(Q,\alpha)_{d^*}^{HN}$, when $d^*$ ranges over HN types of $\alpha$ define a stratification of $R_d$ into irreducible, locally closed subvarieties. If $d=(d^1,...,d^s)$ then the codimension of $Rep(Q,\alpha)_{d^*}^{HN}$ in $Rep(Q,\alpha)$ is given by $-\sum_{1\leq k<l\leq s}\langle d^k,d^l\rangle$.
\end{thm}

\hspace{-13pt}In particular, we also have the HN strata $Rep(\overline{Q},\alpha)_{d^*}^{HN}$ for the double quiver $\overline{Q}$. Notice that $\mu^{-1}(0)\subset Rep(\overline{Q},\alpha)$.
\begin{defn}
    Suppose $\theta\in\mathbb Z^I$ and let $\mu$ be the corresponding moment map defined in section 2.1. Suppose $d^*$ is a tuple with weight $\alpha$, define
    $$\mu^{-1}(0)_{d^*}^{HN}=Rep(\overline{Q},\alpha)_{d^*}^{HN}\cap\mu^{-1}(0)$$
\end{defn}
\subsection{Kac's Polynomial}
Suppose $q$ is a prime power, let $\mathbb F_q$ be the finite field of $q$ elements. Suppose $\alpha\in\mathbb Z^I$, then by we denote $A_{\alpha}(q)$ the number of absolutely indecomposable representations of $Q$ of dimension vector $\alpha$ over $\mathbb F_q$. 
\begin{thm}
    \cite{kac1983root} $A_{\alpha}(q)\neq 0$ if and only if $\alpha$ is a root of the corresponding Kac-Moody Lie algebra $\mathfrak{g}_{Q}$. In that case, $A_{\alpha}(q)$ is a polynomial in $q$, with degree $1-\langle\alpha,\alpha\rangle$.
\end{thm}
\hspace{-13pt}The polynomial $A_{\alpha}(q)$ is called the Kac's polynomial. A remarkable conjecture by Kac states that the coefficient of the Kac's polynomial are nonnegative. This is proven in full generality in \cite{hausel2013positivity}. A special case was proven in \cite{crawley2004absolutely}, for indivisible dimension vectors $\alpha$. (Recall that $\alpha\in\mathbb Z^I$ is indivisible if $\displaystyle \underset{i\in I}{gcd}(\alpha_i)=1$). In particular, the proof in \cite{crawley2004absolutely} gives a cohomological interpretation of the coefficients of the Kac's polynomial.
\begin{thm}
    \cite[Theorem 1.1]{crawley2004absolutely} Let $\lambda\in\mathbb Z^I$ be such that $\lambda\cdot\alpha=0$ but $\lambda\cdot\beta\neq 0$ for $0<\beta<\alpha$. Let $d=1-\langle\alpha,\alpha\rangle$. Then
    $$A_{\alpha}(q)=\sum_{i=0}^d\dim H^{2d-2i}(\mathcal{M}_{\alpha}^{\lambda},\mathbb C)q^i$$
\end{thm}
\begin{rem}
    The stability parameter $\lambda$ that satisfies the condition of Theorem 2.11 is called \emph{generic}.
\end{rem}
\hspace{-12pt}It is interesting to see how the coefficients of the Kac polynomial behave under the variation of inputs. We can do this in two directions, either to vary the quiver or the dimension vector.
\\~\\ For the first direction, in \cite{hennecart2021asymptotic}, Hennecart shows that if we increase the number of edges and let it tend to infinity, then the coefficients of the Kac polynomials stabilise. More precisely, if $n\in\mathbb N^{\Omega}$, we denote $Q_{n}$ by the quiver obtained by replacing the edge $h:i\to j$ with identical edges $n_{h}$. Then
\begin{thm}
    \cite[Theorem 3.2]{hennecart2021asymptotic} Let $m=(m_{\alpha})_{\alpha\in\Omega}\in (\mathbb N\cup\{+\infty\})^{\Omega}$, suppose $\alpha\in\mathbb N^I$, then the sequence of polynomials 
    $$A_{Q_{n},d}\in\mathbb N[q]$$
    converges in $\mathbb N[[q]]$ as $n\to m$. In other words, for any integer $k$ the coefficient in $A_{Q_n,d}(q)$ of $1,q,...,q^k$ stabilizes when $n$ is sufficiently close to $m$.
\end{thm}
\begin{thm}
\cite[Theorem 3.7]{hennecart2021asymptotic} Let $m\in(\mathbb N\cup\{+\infty\})^{\omega}$, suppose $\alpha\in \mathbb N^I$, then the sequence of polynomials 
$$q^{\deg A_{Q_n,d}(q)}A_{Q_n,d}(q^{-1})$$
converges in $\mathbb N[[q]]$ as $n\to m$. In other words, for any integer $k$, the coefficient in $A_{Q_{n,d}}(q)$ of $q^{\deg A_{Q_n,d}(q)},...,q^{\deg A_{Q_n,d}-k}$ stabilizes when $n$ is sufficiently close to $m$.
\end{thm}
\begin{rem}
    For the second approach, see \cite{zveryk2025stabilisation}.
\end{rem}
\begin{rem}
An explicit formula of the Kac's polynomial for all quivers (including those with loops) was provided in \cite[Theorem 4.6]{hua2000counting}. The proof applied Burnside's lemma to the action of the finite group $GL_{\alpha}(\mathbb F_q)$ on the moduli space of quiver representations over $\mathbb F_q$. This will be explained in more detail in section 7.
\end{rem}
\section{Exact Sequence and Lifting}
\hspace{-12pt}Suppose $F^*:0=F^0\subset F^1\subset...\subset F^s$ is a flag of type $d^*$ in the $I$-graded vector space $k^{\alpha}$. Reineke defined in [Rei03] a closed subvariety $Z$ of $R_d$, as the representations in $Rep(Q,\alpha)$ compatible with $F^*$.
\begin{defn} (Also in the proof of [Rei03], Proposition 3.4) We say that a representation $x\in Rep(Q,\alpha)$ is compatible with $F^*$ if for every $1\leq k\leq s$ and edge $\beta:i\to j$ in the quiver, we have 
$$x_{\beta}(F_i^k)\subseteq F_j^k$$
We define $Z$ as the subset of $Rep(Q,\alpha)$ containing all representations that are compatible with $F^*$
\end{defn}
\begin{defn} Similarly, we say that an element $y$ in $End(\alpha)=\displaystyle \prod_{i\in Q}End_{\alpha_i}(k)$ is compatible with $F^*$ if, for all $i\in I$,
$$y_i(F_i^k)\subseteq F_i^k$$
Define $End(\alpha)_{F^*}$ as the subset of representations compatible with $F^*$.
\end{defn}
\begin{defn} Suppose $x\in Rep(Q,\alpha)$, we define $End(x)\subseteq End(\alpha)$ as the vector space of endomorphisms of the representation $x$, and let
$$End(x)_{F^*}=End(x)\cap End(\alpha)_{F^*}$$
\end{defn}
\begin{defn} Let $\mu^{-1}(0)_{F^*}$ be the subvariety of $\mu^{-1}(0)$ consisting of elements that are compatible with $F^*$, considered as a representation in $Rep(\overline{Q},\alpha)$.
\end{defn}
\begin{defn} Suppose $i,j\in I$, and $\beta:i\to j\in\Omega$ we define 
$$Hom(V_i,V_j)_{F^*}\subseteq Hom(V_i,V_j)$$
to be the set of linear maps $f$ in $Hom(V_i,V_j)$ such that for all $1\leq k\leq s$ we have 
$$f(F_i^k)\subseteq F_j^k$$
\end{defn}
\hspace{-12pt}Following [CB01], we have an obvious restriction map 
$$\pi:\mu^{-1}(0)_{F^*}\to Z$$
given by restricting the projection map $Rep(Q,\alpha)\oplus Rep(\overline{Q},\alpha)\to Rep(Q,\alpha)$ to $\mu^{-1}(0)_{F^*}$.
\\~\\ Suppose $x\in Z$, we would like to determine its preimage under $\pi$. We define $End(\alpha)^{T}_{F^*},End(x)_{F^*}^T$and $Z^T$ such that they consist of the transpose of the matrices in the respective space. In particular, $Z^T\subseteq Rep(\overline{Q},\alpha)$, $End(\alpha)^T_{F^*},End(x)_{F^*}^T\subseteq End(\alpha)$. Let $Z^{op},(Z^{op})^T$ be the corresponding definition in the opposite quiver.
\begin{prop} The sequence 
$$Rep(Q^{op},\alpha)_{F^*}\overset{f_1}{\to} End(\alpha)_{F^*}\overset{f_2}{\to} End(x)_{F^*}^{*}\to 0$$
is exact. Here $f_1(\bullet)=\mu(\bullet,x)$, and $f_2(\theta)(\bullet)=\sum_itr(\theta_i\cdot\bullet)$
\end{prop}
\begin{proof}
\fontfamily{cmr}\selectfont
We define a dual flag $D(F^*)$ of $F^*$. For all $i\in I$, we fix a basis 
$$v_1,...,v_{\alpha_i}$$
of $V_i$. This is chosen such that for all $1\leq k\leq s$ (recall that $s$ is the length of our flag), 
$$\{v_1,...,v_{\dim F_i^k}\}$$
is a basis of $F_i^k$. Now we define $D(F^*)$ by
$$D(F^*)^k_{i}=\operatorname{span}\{v_{\dim F_{i}^{s-k}+1},...,v_{\dim V_i}\}$$
We consider the sequence 
$$0\to End(x)_{D(F^*)}\overset{g_1}{\to} End(\alpha)_{D(F^*)}\overset{g_2}{\to} Rep(Q,\alpha)_{D(F^*)}$$
as in [CH98]. The map $g_1$ is just inclusion, and the $g_2$ is given by 
$(\theta_i)_i\mapsto (\theta_{h(\alpha)}x_a-x_a\theta_{t(a)})_a$. It is tautologically exact. Taking the dual sequence, we have 
$$(Rep(Q(\alpha)_{D(F^*)})^*\overset{g_2^*}{\to}End(\alpha)_{D(F^*)}^{*}\overset{g_1^*}{\to}End(x)_{D(F^*)}\to 0$$
We now notice that if $\beta:i\to j\in\Omega$ then
$$\langle\bullet,\bullet\rangle:Hom(V_i,V_j)_{D(F^*)}\times Hom(V_j,V_i)_{F^*}\to \mathbb C,\hspace{20pt}\langle\theta,\phi\rangle= tr(\theta\phi)$$
is a pairing. This is because if $v_1^i,...,v_{\dim V_i}^i$ and $v_1^j,...,v_{\dim V_j}^j$ is the basis chosen for $V_i,V_j$, then $Hom(V_i,V_j)_{D(F^*)}$ is spanned by $e_{kl}:V_i\to V_j,k\leq l$, where $e_{kl}$ sends $v_k^i$ to $v_l^j$ and vanishes on the complement. Meanwhile, $Hom(V_j,V_i)_{F^*}$ is spanned by $f_{lk}:V_j\to V_i$ where $l\geq k$, here $v_{lk}$ sends $v_l^j$ to $v_k^i$. As 
$$\langle e_{pq},f_{rs}\rangle=\delta_{ps}\delta_{qr}$$
This implies that the pairing is nondegenerate, and hence
$$(Rep(Q,\alpha)_{D(F^*)})^*\cong Rep(Q^{op},\alpha)_{F^*}$$
and similarly $$End(x)_{D(F^*)}^*\cong End(x)_{F^*}$$
This gives the exact sequence as claimed; the maps also agree with what we have asserted by the result of [CH98].
\end{proof}
\begin{cor} For all $i\geq 0$ define
$$Rep(Q,\alpha)_{F^*}^i=\{x\in Rep(Q,\alpha)_{F^*}:\dim End(x)_{D(F^*)}^*=i\}$$
also define $C=\dim Rep(Q^{op},\alpha)_{F^*}-\dim End(\alpha)_{F^*}$, then 
$$\dim \pi^{-1}(Rep(Q^{op},\alpha)_{F^*})=C+\max(i+\dim Rep(Q,\alpha)_{F^*}^i)$$
\end{cor}
Recall that $\pi:\mu^{-1}(0)_{F^*}\to Z$ is the restriction map given by restricting $Rep(Q,\alpha)\oplus Rep(\overline{Q},\alpha)\to Rep(Q,\alpha)$ to $\mu^{-1}(0)_{F^*}$
\begin{proof}
\fontfamily{cmr}\selectfont
Indeed, by Proposition 3.6, if $x\in Rep(Q,\alpha)_{F^*}^i$ then \begin{align*}
    \dim\pi^{-1}(x)&=\dim \operatorname{ker}f_1\\
    &=\dim Rep(Q^{op},\alpha)_{F^*}-\dim End(\alpha)_{F^*}+\dim End(x)^*_{D(f)}\\
    &=\dim Rep(Q^{op},\alpha)_{F^*}-\dim End(\alpha)_{F^*}+i\\
    &=C+i
\end{align*}
and hence by the same argument as \cite[Lemma 3.4]{crawley2001geometry},
$$\dim\pi^{-1}(Rep(Q,\alpha)^i_{F^*})=C+i+\dim Rep(Q,\alpha)_{F^*}^i$$
Now since $\pi^{-1}(Rep(Q^{op},\alpha)_{F^*})=\cup_i\pi^{-1}(Rep(Q,\alpha)^i_{F^*})$, we have 
$$\dim\pi^{-1}(Rep(Q^{op},\alpha)_{F^*})=C+\max(i+\dim Rep(Q,\alpha)_{F^*}^i)$$
\end{proof}
\begin{prop} \fontfamily{cmr}\selectfont If $\lambda\in k^I$ then a representation of $Q$ lifts to a representation of $\mu^{-1}(\lambda)$ if and only if the dimension vector $\beta$ of any direct summand satisfies $\lambda\cdot\beta=0$.
\end{prop}
\begin{proof}
\fontfamily{cmr}\selectfont
Same as \cite[Theorem 3.3]{crawley2001geometry}.

\end{proof}
\section{Dimension Computing}
\hspace{-12pt}Our objective is to compute the dimension of the fibres of 
$$\pi:\mu^{-1}(0)_{F^*}\to Rep(Q,\alpha)_{F^*}$$
In view of Corollary 3.7, we need to compute $\max(i+\dim Rep(Q,\alpha)^i_{F^*})$. To do so, we apply a double counting trick which is also used in \cite{kac1983root}.
Define 
$$W=\{(g,x):x\in Rep(Q,\alpha)_{F^*}:g\in End(x)_{F^*}\}\subset End(\alpha)_{F^*}\times Rep(Q,\alpha)_{F^*}$$
Let $\pi_1:W\to End(\alpha)_{F^*},\pi_2:W\to Rep(Q,\alpha)_{F^*}$ be the corresponding projections. We notice that if $g\in V$ then $\pi_1^{-1}(g)$ is a vector space.
\begin{defn} Suppose $r\geq 0$, define 
$$S_r=\{g\in End(\alpha)_{F^*}:\dim\pi_1^{-1}(g)\geq \dim Rep(Q,\alpha)_{F^*}-r\}$$
also define $s_r=\dim S_r$
\end{defn}
\begin{lem} We have 
$$\dim{\pi}^{-1}(Rep(Q^{op},\alpha)_{F^*})=C+\max_r(\dim Rep(Q,\alpha)_{F^*}+s_r-r)$$
\end{lem}
\begin{proof}
We notice that by Corollary 3.3, and the fact that $ Rep(Q,\alpha)^i_{F^*}=\{x\in Rep(Q,\alpha)_{F^*}:\dim\pi_2^{-1}(x)=i\}$
\begin{align*}
    \dim\pi^{-1}(Rep(Q^{op},\alpha)_{F^*})&=C+\max_i(i+\dim Rep(Q,\alpha)_{F^*}^i)\\
    &=C+\dim W\\
    &=C+\max_r(\dim Rep(Q,\alpha)_{F^*}+s_r-r)
\end{align*} 
\end{proof}
\hspace{-12pt}As in \cite{hennecart2021asymptotic}, we let the number of edges in the quiver tend to infinity. That is, let $\Omega$ be the edge set of the quiver and fix $\underline{n}\in \mathbb N^{\Omega}$. Consider the quiver $Q_{\underline{n}}$ which is the quiver $Q$, but with each arrow $\beta\in \Omega$ replaced by $n_{\beta}$ identical arrows. For $\underline{\alpha},\underline{\beta}\in\mathbb N_{\geq 0}^{\Omega}$, we say that $\underline{\alpha}\geq\underline{\beta}$ if for all $h\in \Omega$, we have $\alpha_h\geq\beta_h$. In the following lemma, for the quiver $Q_{\underline{n}}$, we let $s_r^{\underline{n}}$ be the number $s_r$ defined for the quiver $Q_{\underline{n}}$. Let 
\begin{lem} There exists some $\underline{n_0}\in\mathbb Z^{\Omega}$ such that for all $\underline{n}\geq\underline{n_0}$,
$$s_0^{\underline{n}}=\max_{r}(s_r^{\underline{n}}-r)$$
\end{lem}
\begin{proof}
\fontfamily{cmr}\selectfont
Let $\underline{1}\in\mathbb Z^{\Omega}$ be the vector whose component are all $1$. The key observation is that $g\in S_0^{\underline{n}}$ means that $g$ fixes every representation in $Rep(Q_{\underline{n}},\alpha)_{F^*}$, but this holds if and only if $g$ fixes every representation in $Rep(Q_{\underline{1}},\alpha)$. As a result, $S_0^{\underline{1}}=S_0^{\underline{n}}$
for all $\underline{n}\in\mathbb Z^{\Omega}$. Meanwhile, let $n'=\min n_i$, then for all $1\leq r\leq n'-1$,
$$S_{r}^{\underline{1}}\cap S_{rn'-1}^{\underline{n}}=\emptyset$$
In particular, $s_{r}^{\underline{n}}=0$. Let $G=\dim End(\alpha)_{F^*}$, then 
$$\max_{r>0} (s_r^{\underline{n}}-r)\leq G-n'$$ 
Hence it suffices to choose $\underline{n_0}=(G-s_0^{\underline{1}})\underline{1}$.
\end{proof}
\begin{cor} For all $\underline{n}\geq\underline{n_0}$ as above, we have 
$$\dim\pi^{-1}(Rep(Q^{op},\alpha)_{F^*})=C+\dim Rep(Q,\alpha)_{D(F^*)}+s_0$$
\end{cor}
\subsection{Dimension of HN Strata}
Let the HN strata of $\mu^{-1}(0)$ correspond to the dimension type $d^*$ be $T_{d^*}^{HN}$. We will have to modify our previous argument. Let $\varphi:Rep(Q,\alpha)_{F^*}\to R_{d_1}\times ...R_{d_k}$
be the composition of $\pi$ with the projection in \cite[Proposition 3.4]{reineke2003hn}. By \cite{king1994moduli}, $R_{d_1}^{ss}\times...\times R_{d_s}^{ss}$ is an open subvariety of $R_{d_1}\times...\times R_{d_k}$. Let 
$Rep(Q,\alpha)_{F^*}^{ss}=\varphi^{-1}(R_{d_1}^{ss}\times...\times R_{d_s}^{ss})$.
\begin{lem} If $Rep(Q,\alpha)_{F^*}^{ss}$ is nonempty then $\dim Rep(Q,\alpha)_{F^*}^{ss}=\dim Rep(Q,\alpha)_{F^*}$
\end{lem}
\begin{proof}
This follows from the above result by King and the fact that $Z$ is a trivial vector bundle over $R_{d_1}\times...\times R_{d_s}$. \end{proof}
Consider the set 
$$W^{ss}=\{(g,x):x\in Rep(Q,\alpha)_{F^*}^{ss}:g\in End(x)_{F^*}\}\subset End(\alpha)_{F^*}\times Rep(Q,\alpha)_{F^*}^{ss}$$
Using the same argument, we have that for graphs satisfying the condition stated in Lemma 4.2,
$$\dim W^{ss}=s_0+\dim Rep(Q,\alpha)_{F^*}^{ss}=s_0+\dim Rep(Q,\alpha)_{F^*}$$Now we notice that 
$\displaystyle\dim W^{ss}=\dim_{G(\alpha)_F^*}Z^{ss}$. This gives,
$$\dim \pi^{-1}(Rep(Q,\alpha)_{F^*}^{ss})=s_0+\dim Rep(Q,\alpha)_{F^*}+C$$
 Let $\displaystyle\psi=\varphi\circ\pi$. Due to the simple observation that a representation in $Rep(\overline{Q},\alpha)$ is semistable if its image in the projection map to $Rep(Q,\alpha)$ is semistable, we have that 
$\pi^{-1}(Z)\supseteq\mu^{-1}(0)^{ss}_{F^*}\supseteq \psi^{-1}(R_{d_1}^{ss}\times..\times R_{d_k}^{ss})=\pi^{-1}(Z^{ss})$.
Therefore, from $\dim\pi^{-1}(Z)=\dim \pi^{-1}(Z^{ss})=s_0+\dim Rep(Q,\alpha)_{F^*}+C$, we have 
$$\dim \mu^{-1}(0)_{F^*}^{ss}=s_0+\dim Rep(Q,\alpha)_{F^*}+C$$
By Reineke's argument in [Rei03, Proposition 3.4], we have 
\begin{thm}
\begin{align*}\dim T_{d^*}^{HN}&
=s_0+\dim Rep(Q,\alpha)_{F^*}+\dim R_{d^*}^{HN}-\dim End(\alpha)_{F^*}\end{align*}
Here $R_{d^*}^{HN}$ is the $HN$ strata on $Rep(Q,\alpha)$.
\end{thm}
\begin{proof}
Consider the commutative diagram
\begin{center}
\begin{tikzcd}
\mu^{-1}(0)_{F^*}\times_{P}G \arrow[r] & Q_{d^*} \\
\mu^{-1}(0)_{F^*}^{ss}\times_PG \arrow[r, "\cong"] \arrow[u, hook] & T_{d^*}^{HN} \arrow[u, hook]
\end{tikzcd}
\end{center}
We have that the top map is surjective, and 
\begin{align*}
  \dim Q_{d^*}&\leq \dim \mu^{-1}(0)_{F^*}\times_{P}G\\& =\dim \mu^{-1}(0)_{F^*}^{ss}\times_PG \\
  &=\dim T_{d^*}^{HN}
\end{align*}
Hence all these dimensions must be actually equal. As a result,
\begin{align*}\dim T_{d^*}^{HN}&=\dim G(\alpha)-\dim P_{d^*}+s_0+Rep(Q,\alpha)_{F^*}+C\\
&=s_0+\dim Rep(Q,\alpha)_{F^*}+\dim R_{d^*}^{HN}-\dim End(\alpha)_{F^*}\end{align*}
\end{proof}
\begin{rem} It is tempting to find a general formula for $\dim T_{d^*}^{HN}$. We notice that in the proof of the dimension of absolutely decomposable representations in \cite{kac1983root} uses the following fact:
\begin{enumerate}
    \item[(i)] There are finitely many unipotent conjugacy classes in $GL_{\alpha}(\mathbb C)$.
    \item[(ii)] There is a nice representative for each of these conjugcay classes (given by Jordan canonical form).
\end{enumerate}
However, to generalise this method to our present setting, we will need similar properties for the parabolic subgroup $Rep(Q,\alpha)_{F^*}$. However, this is not true for arbitrary $F^*$. Indeed, by \cite{hille1999parabolic}, $Rep(Q,\alpha)_{F^*}$ has finitely many unipotent conjugacy classes if and only if $F^*$ has at most five terms. In the case of maximal parabolic subgroups (i.e. when $F^*$ has only two terms), a representative of conjugacy classes (for small dimensions) was constructed in \cite{murray2000conjugacy}. However, it is still significantly more complicated than the Jordan canonical form and is difficult to use in computations.
\end{rem}
\section{One-parameter Subgroup and Improving the Bound}
\begin{lem} Suppose $V,W$ are two vector spaces, for $A\in End(V)_{F^*}$ and $B\in End(W)_{F^*}$ define 
$$r(A,B)=\{f\in Hom(V,W)_{F^*}:fA=Bf\}$$
then the function $c:End(V)_{F^*}\times End(W)_{F^*}\to \mathbb Z$ given by 
$$c(A,B)=\dim r(A,B)$$
is upper-semicontinuous in the Zariski topology.
\end{lem}
\begin{proof}
Define $r_{A,B}:Hom(V,W)_{F^*}\to Hom(V,W)_{F^*}$ given by $r_{A,B}(f)=\operatorname{rank}fA-Bf$, then 
$$(A,B)\mapsto rank(r_{A,B})$$
is lower-semicontinuous since the inverse image of $(-\infty,n]$ is given by vanishing of determinants. Hence 
$$c_{A,B}=\dim Hom(V,W)_{F^*}-r_{A,B}$$
is upper-semicontinuous. 
\end{proof}
\begin{defn} Suppose $F^*:\{0\}=F^0\subset F^1\subset...\subset F^s=\oplus V_i $ is a flag, then we let 
$$\epsilon_{F^*}=\min_{i\in I}\{\dim F^1_i,\operatorname{codim}_{V_i}F_i^{s-1}\}$$
If $d^*$ is a partition of $\alpha$, define $\epsilon_{d^*}=\epsilon_{F^*}$, where $F^*$ is a flag of type $d^*$.
\end{defn}
\begin{ex} if we have the $A_1$ quiver, dimension vectors $(10,15)$ and the flag $\mathbb F^*$ is given by 
$$\{0\}\oplus\{0\}\subset \mathbb C^3\oplus \mathbb C^5\subset \mathbb C^4\oplus \mathbb C^{11}\subset \mathbb C^{10}\oplus C^{15}$$
then $\epsilon_{F^*}=3$.
\end{ex}
\begin{thm} If $\epsilon_{F^*}\geq n$, then $s_0=1$ and $s_l=0$ for all $1\leq l\leq n-1$ as in definition 4.1.
\end{thm}
\begin{proof}
We proceed by induction on $n$. The case $n=0$ is evident.
\\~\\ Suppose all smaller cases hold, then by inductive hypothesis it suffices to prove $s_{n-1}\leq1$, if $(g_i)_{i\in I}\in S_{n-1}$, then by inductive hypothesis there exists $i,j\in I$ such that in the notation of Lemma 5.1, $c(g_i,g_j)\geq \dim Hom(V_i,V_j)_{F^*}$. 
We first notice that the element whose $i$ component is $\lambda \operatorname{id}$ for every $i\in I$ belongs to $S_0$ and hence all of $S_1,...,S_{n-1}$, hence $s_0,...,s_{n-1}\geq 1$. Therefore, it suffices to show that $\beta:i\to j$ is an edge, then for $A\in End(V)_{D(F^*)},B\in End(W)_{D(F^*)}$ where $(A,B)\neq (\lambda \operatorname{id},\lambda\operatorname{id})$, then we must have
$$c(A,B)\leq \dim Hom(V,W)_{F^*}-\epsilon_{F^*}\hspace{20pt}(1)$$
let $A=l_Au_A$ and $B=l_Bu_B$ be its Levi decomposition in its respective parabolic subalgebra $End(V)_{F^*},End(W)_{F^*}$. We notice that if we conjugate $A$ by some $p=lv\in End(V)_{F^*}$, then the levi decomposition of $pAp^{-1}$ will be $ll_Al^{-1}$. Together with the fact that $c(pAp^{-1},B)=c(A,B)$ we can WLOG assume that $A$, and similarly $B$ are both lower triangular matrices.
\\~\\ Now we let $D_A$ and $D_B$ be the set of diagonal entries of $A,B$ respectively. We divide into two cases:
\\~\\\underline{Case I:} $|D_{A}\cup D_B|\geq 2$
\\ Pick positive integers $a_1>a_2>...>a_{\dim V}$. Define, for all $t\in\mathbb C$, matrices $M_t=\operatorname{diag}(t^{a_1},...,t^{a_n})$. Consider the mapping $\lambda:\mathbb C^*\to End(V)_{*}$ 
$$t\mapsto M_tAM_t^{-1}$$
We notice that 
$$(M_tAM_t^{-1})_{ij}=(t^{a_i-a_j}A)_{ij}$$
By the assumption that $A$ is upper triangular, we notice that the matrix $d(A)$ consisting only of the diagonal entries of $A$ lies in $\lambda(\mathbb C^*)$, meanwhile, we have that $c(M_tAM_t^{-1},B)=c(A,B)$.
Therefore, by Lemma 5.1, we have 
$$c(d(A),B)\geq c(A,B)$$
Hence, we may WLOG assume $A$ and similarly $B$ are diagonal matrices. Now  either
\begin{enumerate}
    \item[(i)] For each $1\leq j\leq \epsilon_{F^*}$ there exists $1\leq i\leq \dim V$ such that $a_{i,i}\neq b_{j,j}$, or
    \item[(ii)] for each $1\leq i\leq \epsilon_{F^*}$ there exists some $1\leq j\leq \dim W$ such that $a_{\dim V-i,\dim V-i}\neq b_{j,j}$
\end{enumerate} 
In the former case, the map sending $v_i$ to $w_j$ is not in $r(A,B)$, and in the latter case, the map sending $v_{\dim V-i}$ to $w_j$ is not in $r(A,B)$. However, both of these are in $Hom(V,W)_{F^*}$ by the definition of $\epsilon{F^*}$. Hence, the codimension of $r(A,B)$ is at least $d$.
\\~\\ \underline{Case II}: $|D_A\cup D_B|=1$
\\ This is equivalent to all the diagonal entries of $A,B$ being the same. Therefore, we may write $A=\lambda I+A',B=\lambda I+B'$. Notice that $\lambda I$ commutes with all maps, so $$c(A,B)=c(A',B')$$ By our assumption, at least one of $A',B'$ must be nonzero. Suppose $A'\neq 0$, then by the $\mathbb C^*$ action argument in Case I, we may assume $B'=0$. Then we notice that $fA=0$ if the image of $A$, which is nonzero by assumption, lies in the kernel of $f$. This again shows that the codimension of $r(A,B)$ is at least $\epsilon_{F^*}$. 
\\~\\ Similarly, if $B'\neq 0$, then we can assume $A'=0$, then $B'f=0$ implies that $f$ must map into the kernel of $B'$, which is not $W$ by assumption, and hence the codimension is at least $\epsilon_{F^*}$ by the same argument.
\end{proof}
\begin{cor} Suppose $\epsilon_{d^*}\geq 1$. If $\underline{n}$ satisfies $\underline{n}\geq\displaystyle\frac{End(\alpha)_{d^*}}{\epsilon_{d^*}}\underline{1}$, then 
$$\dim T_{d^*}^{HN}=1+\dim R_{d^*}^{HN}+C$$
\end{cor}
\begin{proof}
Following the notation as in Lemma 4.2, for all $1\leq a\leq n'\epsilon_{d^*}-1$,
$$S_{a}^{m\underline{1}}=\emptyset$$
Hence 
$$\max_{r>0}(s_{r}^{\underline{n}}-r)\leq G-\epsilon_{d^*}n'$$
Hence the result follows from the fact that $s_0=1$.
\end{proof}
\subsection{Codimension and Transverse Intersection}
We continue the assumption in Corollary 5.5. Let $\overline{R}_{d^*}^{HN}$be the $HN$ strata of the quiver $\overline{Q}$. We notice that $\mu^{-1}(0)\cap \overline{R}_{d^*}^{HN}=T_{d^*}^{HN}$. We hence have the diagram

\begin{center}
\begin{tikzcd}
\mu^{-1}(0) \arrow[r, hook]  & Rep(\overline{Q},\alpha)  \\
T_{d^*}^{HN}\arrow[u, hook] \arrow[r, hook] & R_{d^*}^{HN}\arrow[u, hook]
\end{tikzcd}
\end{center}
It is natural to consider the codimension of the two vertical arrows.
By [CB01], we have 
$$\dim \mu^{-1}(0)=\alpha\cdot\alpha+1-2\langle\alpha,\alpha\rangle_{Q}$$
Meanwhile, by [Rei03, Proposition 3.4], we have 
\begin{align*}
\dim \overline{R}_{d^*}^{HN}-\dim R_{d^*}^{HN}&=\dim Rep(\overline{Q},\alpha)_{F^*}-\dim Rep(Q,\alpha)_{F^*}\\
&=\dim Rep(Q^{op},\alpha)_{F^*}
\end{align*}
Also recall that $C=\dim Rep(Q^{op},\alpha)_{F^*}-\dim End(\alpha)_{F^*}$. 
As a result,
\begin{align*}
\operatorname{codim}_{Rep(\overline{Q},\alpha)}\mu^{-1}(0)&=2(-\langle\alpha,\alpha\rangle_{Q}+\alpha\cdot\alpha)-(\alpha\cdot\alpha+1-2\langle\alpha,\alpha\rangle_Q)\\
&=\alpha\cdot\alpha-1\\
&=\dim End(\alpha)-1
\end{align*}
and 
\begin{align*}
\operatorname{codim}_{\overline{R}_{d^*}^{HN}}T_{d^*}^{HN}&=\dim \overline{R}_{d^*}^{HN}-\dim R_{d^*}^{HN}-1-C\\
&=\dim Rep(Q^{op},\alpha)_{F^*}-1-C\\
&=End(\alpha)_{F^*}-1
\end{align*}
Therefore, 
\begin{align*}    \operatorname{codim}_{Rep(\overline{Q},\alpha)}R_{d^*}^{HN}-\operatorname{codim}_{\mu^{-1}(0)}T_{d^*}^{HN}&=\operatorname{codim}_{Rep(\overline{Q},\alpha)}\mu^{-1}(0)-\operatorname{codim}_{\overline{R}_{d^*}^{HN}}T_{d^*}^{HN}\\
    &=\dim End(\alpha)-\dim End(\alpha)_{F^*}
\end{align*}
\section{Hua's Formula}
\subsection{Strata and Cohomology}
We recall the following correspondence between Harder-Narasimhan stratification of quiver representation and the Hesselink stratification, as explained in \cite{hoskins2018stratifications}:  In general, suppose $G$ is a reductive group acting on a scheme $Y$ with respect to a linearisation $L$. If one of the conditions in \cite[Section 2]{hoskins2018stratifications} hold, and we have a norm $||\bullet||$ on the conjugacy classes of $1$-PSs of $G$, then we have a stratification 
$$Y\setminus Y^{ss}=\bigsqcup_{\beta}$$
which is a stratification of $Y$ into finitely many $G$-invariant locally closed subschemes. By \cite[Remark 2.10]{hoskins2018stratifications}, this stratification is indexed by conjugacy classes of rational $1$-PS. 
\\~\\ Specialising to the setting of quiver representations, suppose $\alpha\in\mathbb Z^I$. Let $\chi_{\theta}$ be a character of $GL_{\alpha}$ arising from $\theta\in\mathbb Z^I$ as in section 2.1, it corresponds to a linearisation of $Rep(Q,\alpha)$. We have the following norm constructed in \cite[Example 2.4]{hoskins2018stratifications}: For any $n$, fix a maximal torus $T$ of $GL_{n}$ and a Weyl-invariant norm $||\bullet||_T$ on $X_{*}(T)_{\mathbb R}$. This is extended to $X_{*}(GL_{n}(\mathbb C))$ via the conjugacy action of $GL_{n}(\mathbb C)$. For the product $GL_{\alpha}$, define
$$||(\lambda_1,...,\lambda_r)||_{\alpha}^2:=\sum_{i=1}^r\alpha_i||\lambda_i||^2$$
Now for a HN type $d^*=(d_1,...,d_s)$, let $r_i=-(\theta\cdot d_i)/(\alpha\cdot d_i)$, 
suppose $v$ is a vertex in the quiver, the $1$-PS for $GL_{\alpha_v}$ is given by 
$$\lambda_{\gamma,v}'(t)=\operatorname{diag}(t^{r_1},...,t^{r_1},t^{r_2},...,t^{r_2},...,t^{r_s},...,t^{r_s})$$
These patch together to form a rational $1$-PS $\lambda_{d^*}$ on $GL_{\alpha}(\mathbb C)$. The main theorem in \cite{hoskins2014stratifications},\cite{hoskins2018stratifications} states:
\begin{thm}
\cite[Theorem 3.8]{hoskins2018stratifications} Let $(Q,R)$ be a quiver with relations with dimension vector $d\in\mathbb N^I$, then the HN-stratification with respect to $(\theta,\alpha)$.  and the Hesselink stratification with respect to $\chi_{\theta}$ and $||\bullet||_{\alpha}$ conicide. That is, if 
$$Rep(Q,\alpha,R)=\bigsqcup_{d^*}R_{d^*}\text{ and }Rep(Q,\alpha,R)=\bigsqcup_{[\gamma]}S_{[\gamma]}$$
then we have $R_{d^*}=S_{\lambda_{d^*}}$.
\end{thm}
\hspace{-12pt}In particular, we have the following:
\begin{cor}
    If $\mu^{-1}(0)_{d^*}^{HE}$ denotes the Hesselink strata corresponds to $\lambda_{d^*}$, then
    $$\mu^{-1}(0)_{d^*}^{HE}=T_{d^*}^{HN}$$
\end{cor}
\hspace{-12pt}Intuitively, in a smooth variety, removing a strata with high codimension should not affect the low degree cohomologies. In particular, we have the following result
\begin{thm}
    \cite{kirwan1984cohomology} If $X$ is a nonsingular variety, then
    $$\dim H^n_G(X,\mathbb C)=\sum_{\lambda}\dim H_G^{n-d(\lambda)}(S_{\lambda},\mathbb C)$$
    where $d(\lambda)$ is the real codimension of $S_{\beta}$ in $X$.
\end{thm}
\hspace{-12pt}Indeed, $\mu^{-1}(0)$ is singular and we will try to get a similar formula in the next subsection.
\subsection{Hennecart's Linear Form} Again suppose $Q=(I,\Omega)$ is a quiver and suppose $\alpha\in\mathbb Z^I$, we say that $(\alpha^1,...,\alpha^l)\in \mathbb (Z^I)^l$ is a partition of $\alpha$ if $\alpha=\alpha^1+...+\alpha^l$.  Suppose $\alpha,\beta\in\mathbb Z^I$, we say that $\alpha\geq\beta$ if $\alpha-\beta\in\mathbb Z_{\geq 0}^I$ and $\alpha\neq \beta$. An ordered partition of $\alpha$ is a partition $\alpha^1,...,\alpha^l$ such that $\alpha^1\geq...\geq \alpha^l$. A double partition of $\alpha$ is a partition $(\alpha^1,...,\alpha^l)$ of $\alpha$, and an ordered partition $(d^{i,1},...,d^{i,s})$ of $\alpha^i$ for each $1\leq i\leq l$. Let $P(\alpha)$ be the set of double partitions of $\alpha$.
\\~\\ Recall that $\theta$ is a stability parameter for $\alpha$ such that $\theta\cdot\alpha=0$ and $\theta\cdot\beta<0$ for all $\beta<\alpha$, and $\mu$ is the associated slope function. Following \cite{reineke2003hn}, we say that a partition $d^*=(d^1,...,d^l)$ of $d$ is a HN-type for $\mu^{-1}(0)$ if $\mu^{-1}(0)_{d^*}^{HN}$ is nonempty.
\begin{lem}
    There exists $\underline{m}\in\mathbb Z^I$ such that for all $\underline{n}\geq\underline{m}$, $d^*=(d^1,...,d^l)$ is a HN-type if and only if 
    $$\mu(d^1)>...>\mu(d^l)$$
\end{lem}
\begin{proof}
    By definition if $d^*$ is a $HN$-type then $\mu(d^1)>\mu(d^2)>...>\mu(d^l)$.
    \\~\\ Conversely, suppose $\mu(d^1)>\mu(d^2)>...>\mu(d^l)$. By the result in section 5 and the same argument as \cite[Corollary 3.5]{reineke2003hn}, there exists a semistable representation of $\mu^{-1}(0)$ for the dimension vector $d^i$ if for all partition $d^i=\beta_1+...+\beta_k$,
    $$\dim End(\beta_i)-\sum_{i<j}\langle \beta_i,\beta_j\rangle\geq 0 $$
    However, as we increase the edges, the Euler form term in the above expression tends to infinity.
    \\~\\ Therefore, for sufficiently large $\underline{n}$, we can pick $s_1,...,s_l$ such that $s_i\in \mu^{-1}(0)$ for the dimension vector $d^i$. Then $s_1\oplus...\oplus s_l$ will be an element in $\mu^{-1}(0)$ for $Rep(Q,\alpha)$ with HN-type $d^*$.
    
\end{proof}
\hspace{-12pt}Now let $R(\alpha)$ denote the set of double partitions such that $\mu(d^{i,j})$ are pairwise distinct, and let $S(\alpha)$ denote the set of double partitions such that some two $\mu(d^{i,j})s$ are equal.
We can naturally associate to each ordered partition a partition of $\alpha$. We simply order all $d^{i,j}$ by their value under $\mu$. By Lemma 7.4, elements of $R(\alpha)$ correspond to HN types under this association. Let $T(\alpha)$ be the set of $HN$-types and $U(\alpha)$ be the set of partition $(d^1,...,d^l)$ of $\alpha$ such that $\mu(d^1)\geq...\geq \mu(d^l)$ but some of the inequalities are not strict.
\\~\\
We recall the Hua's Formula, we will first need to define a function $\phi$. Suppose $\pi$ is a partition of an integer $n$ given by $n=n_1+...+n_k$, define $\phi_{\pi}(q)=\prod_{i=1}^k\prod_{j=1}^{n_j}(1-q^i)$. Suppose we are given a family of partitions $\pi=(\pi_i)_{i\in I}$ indexed by the vertex set $I$ of the quiver $Q$, we define
$$\phi(\pi)(q)=\prod_{i\in I}\phi_i(q)$$

\begin{thm} \cite[Theorem 4.6]{hua2000counting},\cite[Theorem 4.1]{zveryk2025stabilisation} Let $\alpha$ be an indivisible vector. $A_{\alpha}(q)$ is given by the following formula.
$$A_{\alpha}(q)=(q-1)\sum_{l=1}^{\infty}\frac{(-1)^{l+1}}
{l}\sum_{\substack{(\alpha^1,...,\alpha^l)\\\alpha=\alpha^1+...+\alpha^l}}\prod_{i=1}^l\left[\sum_{\substack{\alpha^i=d^{i,1}+...+d^{i,s}\\d^{i,1}\geq...\geq d^{i,s}}}\frac{q^{-\sum_k\langle d^{i,k},d^{i,k}\rangle}}{\prod_k\phi_{d^{i,k}-d^{i,k+1}}(q^{-1})}\right]$$
\end{thm}
\hspace{-12pt} Suppose $\gamma\in\mathbb N_{\geq 0}^I$, we define 
$$b(\gamma)=\sum_{i\in I}\frac{\gamma_i(\gamma_i+1)}{2}$$
Then we can equivalently write the Hua's formula as
$$A_{\alpha}(q)=(q-1)\sum_{l=1}^{\infty}\frac{(-1)^{l+1}}
{l}\sum_{\substack{(\alpha^1,...,\alpha^l)\\\alpha=\alpha^1+...+\alpha^l}}\prod_{i=1}^l\left[\sum_{\substack{\alpha^i=d^{i,1}+...+d^{i,s}\\d^{i,1}\geq...\geq d^{i,s}}}(-1)^s\frac{q^{-\sum_k\langle d^{i,k},d^{i,k}+b(d^{i,k}-d^{i,k+1})\rangle}}{\prod_k\phi_{d^{i,k}-d^{i,k+1}}(q)}\right]$$
\begin{lem}
Suppose $\pi\in P(\alpha)$, corresponding to $\alpha=\alpha^1+...+\alpha^l$ and an ordered partition $(d^{i,1},...,d^{i,s})$ for each $\alpha^i$, then the polynomial $\prod_i\prod_k\phi_{d^{i,k}-d^{i,k+1}}(q)$ divides $\phi_{\alpha}(q)$.
\end{lem}
\begin{proof}
We can assume that all order partitions of $\alpha^i$ is trivial. In that case, we only need to prove that if $\alpha=\alpha^1+...+\alpha^l$, then 
$\phi_{\alpha}(q)$ is a multiple of $\phi_{\alpha^1}(q)...\phi_{\alpha^l}(q)$.
\\~\\ However, as $q^a-1\mid q^b-1$ if $a\mid b$, this follows easily from the fact that for any integer $n$ and $n=n_1+...+n_k$,
$$\frac{n!}{(n_1)!...(n_k)!}$$
is an integer, which is a consequence of the generalized binomial theorem.
\end{proof}
\hspace{-13pt}As a result, we have
$$\frac{\phi_{\alpha}(q)A_{\alpha}(q)}{q-1}=\sum_{d^*\in T(\alpha)}P_{d^*}(q)+\sum_{d^*\in U(\alpha)}P_{d^*}(q)$$
where $P_{d^*}(q)$ is a polynomial in $q$. For each $d^*\in T(\alpha)$, let $l(d^*)$ be the degree of $P_{d^*}(q)$, and write $P_{d^*}(q)=q^{l(d^*)}P_{d^*}'(q)$. Recall the notation of $Q_{\underline{n}}$ in section 4.1. We notice that $P_{d^*}'(q)$ does not depend on $\underline{n}$ by the construction of the Hua's formula. 

\begin{thm}Let $d^*\in T(\alpha)$ corresponding representation type of $\alpha$, then if the condition in Corollary 5.5 is satisfied, then
$$l(\alpha)-l(d^*)=\text{codim}_{\mu^{-1}(0)}\mu^{-1}(0)_{d^*}^{HN}$$
\end{thm}
\begin{proof}
By the result in section 5.1, we have that 
\begin{align*}
    \text{codim}_{\mu^{-1}(0)}\mu^{-1}(0)_{d^*}^{HN}&=\text{codim}_{Rep(\overline{Q},\alpha)}\overline{R}_{d^*}^{HN}-\dim End(\alpha)+\dim End(\alpha)_{F^*}\hspace{20pt}(1)
\end{align*}
By [Rei03],
$$\text{codim}_{Rep(\overline{Q},\alpha)}\overline{R}_{d^*}^{HN}=-\sum_{k<l}\langle d^k,d^l\rangle_{\overline{Q}}$$
Meanwhile, suppose $\pi\in P(\alpha)$ is given by a partition $\alpha=\alpha^1+...+\alpha^l$ and an ordered partition $\alpha^i=d^{i,1}+...+d^{i,s}$ for each $\alpha^i$, then the degree of 
$$\frac{\phi_{\alpha}(q)}{q-1}\cdot \frac{q^{-\sum_{k}(\langle d^{i,k},d^{i,k}\rangle+b(d^{i,k}-d^{i,k+1}))}}{\prod_{k}\phi_{d^{i,k}-d^{i,k+1}}(q)}$$ is given by 
\begin{align*}
    &\hspace{13pt}\deg \phi_{\alpha}(q)-\sum \deg\phi_{d^{i,k}-d^{i,k+1}}(q)+\sum\deg \phi_{d^{i,k},d^{i,k+1}}(q)-\sum_{i}\sum_{k}\langle d^{i,k},d^{i,k}\rangle\\&=\deg \phi_{\alpha}(q)-\sum_{i}\sum_{k}\langle d^{i,k},d^{i,k}\rangle
\end{align*}
Therefore, it suffices to show that if $\alpha=\alpha^1+...+\alpha^l$, which corresponds to a HN type $d^*$ and a flag $F^*$ then 
$$-\langle\alpha,\alpha\rangle+\sum_{i=1}^l\langle \alpha^i,\alpha^i\rangle=-\sum_{1\leq m<n\leq l}\langle \alpha^i,\alpha^i\rangle_{\overline{Q}}-\dim End(\alpha)+\dim End(\alpha)_{F^*}$$
To show this we notice that
\begin{align*}
    -\langle\alpha,\alpha\rangle+\sum_{i=1}^l\langle\alpha^i,\alpha^i\rangle&=-\sum_{1\leq m<n\leq l}(\langle\alpha^m,\alpha^n\rangle_{Q}+\langle\alpha^m,\alpha^n\rangle_{Q^{op}})\\
    &=-\sum_{1\leq m<n\leq l}(\langle\alpha^{m},\alpha^n\rangle_{\overline{Q}}+\alpha^m\cdot\alpha^n)\\
    &=-\sum_{1\leq m<n\leq l}\langle\alpha^{m},\alpha^n\rangle-\dim End(\alpha)+\dim End(\alpha)_{F^*}
\end{align*}
as desired. \end{proof}
Therefore, combining Theorem 6.5, we have
\begin{thm}
$$\frac{\phi_{\alpha}(q)}{q-1}A_d(q)=\sum_{d^*\in T(\alpha)}q^{-\operatorname{codim}T_{d^*}^{HN}}P'_{d^*}(q^{-1})+\sum_{d^*\in U(\alpha)}P_{d^*}(q)$$
where $P'_{d^*}$ and $P_{d^*}(q)$ are polynomials independent of $\underline{n}$. 
\end{thm}
\begin{rem}
    We first note that if $G=PGL_{\alpha}(\mathbb C)$, then
    $$H_{G}^{\bullet}(pt)=\frac{q-1}{\phi_{\alpha}(q)}$$
    In view of Theorem 2.11, $A_{d}(q)$ can be considered as the even dimensional equivariant Poincare series for $\mu^{-1}(0)^{ss}$. Hence Theorem 6.8 shows that if we remove all stratum except the semistable one, then the even cohomologies of low degree (up to the codimension of the second largest strata) is unchanged. This is in similar spirit with Kirwan's Formula (Theorem 6.3).
\end{rem}
\section*{Acknowledgments}
The author is grateful to Vladyslav Zveryk for introducing him this problem and Ivan Losev and Vladyslav Zveryk for helpful comments and discussions. The author thanks also Chinese University of Hong Kong for supporting his visit to Yale University.


\end{document}